\documentclass[12pt,a4paper]{amsproc}

\newtheorem{theorem}{\bf Theorem}[section]
\newtheorem{lemma}[theorem]{\bf Lemma}
\newtheorem{proposition}[theorem]{\bf Proposition}
\newtheorem{corollary}[theorem]{\bf Corollary}

\theoremstyle{remark}

\author{Sara Rodrigues}
\address{ Department of Mathematics, University of Brasilia, Brasilia-DF, 70910-900 Brazil}
\email{sararaissa@mat.unb.br}
\author{Pavel Shumyatsky}
\address{Department of Mathematics, University of Brasilia, Brasilia-DF, 70910-900 Brazil}
\email{pavel@unb.br}

\thanks{This work was supported by the Conselho Nacional de Desenvolvimento Cient\'{\i}fico e Tecnol\'ogico (CNPq),  and Funda\c c\~ao de Apoio \`a Pesquisa do Distrito Federal (FAPDF), Brazil.}

\title[Exponent of a finite group]
{Exponent of a finite group admitting a coprime automorphism}

\begin{document}

\begin{abstract}
\noindent  
Let $G$ be a finite group admitting a coprime automorphism $\phi$ of order $n$. Denote by $G_{\phi}$ the centralizer of $\phi$ in $G$ and by $G_{-\phi}$ the set $\{ x^{-1}x^{\phi}; \ x\in G\}$. We prove the following results.

\noindent 1. If every element from $G_{\phi}\cup G_{-\phi}$ is contained in a $\phi$-invariant subgroup of exponent dividing $e$, then the exponent of $G$ is $(e,n)$-bounded.

\noindent  2. Suppose that $G_{\phi}$ is nilpotent of class $c$. If $x^{e}=1$ for each $x \in G_{-\phi}$ and any two elements of $G_{-\phi}$ are contained in a $\phi$-invariant soluble subgroup of derived length $d$, then the exponent of $[G,\phi]$ is bounded in terms of $c,d,e,n$.   
\end{abstract}
\maketitle

\section{Introduction}

Given a group $G$ with an automorphism $\phi$, denote by $G_\phi$ the fixed-point subgroup $\{x\in G; \ x^\phi=x\}$ and by $G_{-\phi}$ the set $\{x^{-1}x^\phi;\ x\in G\}$. Write $[G,\phi]$ for the subgroup generated by $G_{-\phi}$. The following theorem was proved in \cite{asa}.

\begin{theorem}\label{0} Let $c,d,e$ be nonnegative integers and $G$ a finite group of odd order admitting an involutory automorphism $\phi$ such that $G_\phi$ is nilpotent of class $c$ and $x^e=1$ for each $x\in G_{-\phi}$. Suppose that the subgroup $\langle x,y\rangle$ has derived length at most $d$ for every $x,y\in G_{-\phi}$. Then the exponent of $[G,\phi]$ is $(c,d,e)$-bounded.
\end{theorem}

Throughout the article we use the expression ``$(a,b,c\dots)$-bounded" to mean ``bounded from
above by some function depending only on the parameters $a,b,c\dots$". Recall that the exponent of a finite group $G$ is the minimum number $e$ such that $x^e=1$ for all $x\in G$.

The present paper will be concerned with the question of possible extension of the above theorem to the case where the automorphism has order greater than two. One important tool used in the proof of Theorem \ref{0} is the fact if $G$ is a group of odd order admitting an involutory automorphism $\phi$ such that each element in $G_\phi\cup G_{-\phi}$ has order dividing $e$, then the exponent of $G$ is $e$-bounded. This was proved in \cite{shu11} using the techniques created by Zelmanov in his solution of the restricted Burnside problem \cite{zelm} (see Section 3 of the present article for details).

Recall that an automorphism $\phi$ of a finite group $G$ is called coprime if $(|G|,|\phi|)=1$. Assume that a finite group $G$ admits a coprime automorphism $\phi$ of order $n$ such that  each element in $G_\phi\cup G_{-\phi}$ has order dividing $e$. It is unknown whether the exponent of $G$ can be bounded in terms of $e$ only. We will establish the following theorem.

\begin{theorem}\label{1} Let $e$, $n$ be positive integers, and let $G$ be a finite group admitting a coprime automorphism $\phi$ of order $n$ such that each element from $G_\phi\cup G_{-\phi}$ is contained in a $\phi$-invariant subgroup of exponent dividing $e$. Then the exponent of $G$ is $(e,n)$-bounded.
\end{theorem}

Note that in the case where $n=2$ we have $x^\phi=x^{-1}$ for all $x\in G_{-\phi}$ and therefore any subgroup generated by a subset of $G_{-\phi}$ is $\phi$-invariant. We use Theorem \ref{1} to prove the following result, which is an extension of Theorem \ref{0}.

\begin{theorem} \label{2} Let $G$ be a finite group admitting a coprime automorphism $\phi$ of order $n$ such that $G_{\phi}$ is nilpotent of class $c$ and $x^{e}=1$ for each $x\in G_{-\phi}$. Suppose that any two elements of $G_{-\phi}$ are contained in a $\phi$-invariant soluble subgroup of derived length $d$. Then the exponent of $[G,\phi]$ is bounded in terms of $c,d,e,n$.
\end{theorem}

In the next section we discuss some general properties of finite groups admitting a coprime automorphism. Section 3 provides background information for applications of Lie methods in group theory. This is used in Section 4 to prove Theorem \ref{1}. In Section 5 we prove Theorem \ref{2}.

\section{Preliminaries}

The first lemma is a collection of well-known facts about coprime automorphisms (see for example \cite[6.2.2, 6.2.4]{go}). 

\begin{lemma}\label{000} Let $\phi$ be a coprime automorphism of a finite group $G$.
\begin{enumerate}
\item $[G,\phi]=[G,\phi,\phi]$;
\item If $N$ is a $\phi$-invariant normal subgroup of $G$, we have $(G/N)_\phi=G_\phi N/N$;
\item  If $H$ is a $\phi$-invariant $p$-subgroup of $G$, then $H$ is contained in a $\phi$-invariant Sylow $p$-subgroup of $G$.
\end{enumerate}
\end{lemma}

We will also require the following fact.

\begin{lemma}\label{0000} Let $\phi$ be a coprime automorphism of a finite group $G$.
If $N$ is a $\phi$-invariant normal subgroup of $G$ contained in $G_{\phi}$, then $[G,\phi]$ centralizes $N$.
\end{lemma}
\begin{proof} Let $x\in N$ and $y\in G$. Both elements $x$ and $x^y$ lie in $G_\phi$. We have $x^y=(x^y)^\phi=x^{y^\phi}$, whence $x=x^{yy^{-\phi}}$. Since elements of the form $yy^{-\phi}$ generate $[G,\phi]$, the lemma follows.
\end{proof}
It is well-known that if $G$ is of odd order and $\phi$ is of order two, then $G=G_{-\phi}G_\phi$. We will now examine the question whether the equality $G=G_{-\phi}G_\phi$ holds for
any coprime automorphism $\phi$. In what follows we will see that in general this equality 
may fail. We however will establish that the equality $G=G_{-\phi}G_\phi$ does verify if $G$
is nilpotent.

\begin{lemma}\label{001} Let $\phi$ be a coprime automorphism of a finite group $G$. Then
$G=G_{-\phi}G_\phi$ if and only if no nontrivial element of $G_{-\phi}$ has conjugates in $G_\phi$.
\end{lemma}
\begin{proof} Suppose first that $G=G_{-\phi}G_\phi$.
Since $|G|=|G_{-\phi}||G_\phi|$, it follows that each
element $x\in G$ can be written uniquely in the
form $x=gh$, where $g\in G_{-\phi}$ and $h\in G_\phi$.
Suppose that there exist elements $1\neq a\in G_\phi$
and $b,c\in G$ such that $$b^{-1}b^\phi=c^{-1}ac.$$
This implies $$c^{\phi^{-1}}b^{-1}b^\phi c^{-1}=
c^{\phi^{-1}}c^{-1}a.$$ Therefore there exist at least
two ways to write the element $x=c^{\phi^{-1}}b^{-1}
b^\phi c^{-1}$ in the form $x=gh$ with $g\in G_{-\phi}$
and $h\in G_\phi$, a contradiction.

Assume now that no nontrivial element of $G_{-\phi}$
has conjugates in $G_{\phi}$. We want to prove that
$G=G_{-\phi}G_\phi$. Suppose that this is false and
so, since $|G|=|G_{-\phi}||G_\phi|$, we
have $gh=g_1h_1$ for some pairwise distinct elements
$g,g_1\in G_{-\phi}$ and $h,h_1\in G_{\phi}$. 
It follows that there exist elements $x,y\in G$ and $1 \neq a
\in G_{\phi}$ such that $x^{-1}x^\phi=y^{-1}y^\phi a$.
Then $yx^{-1}x^\phi y^{-\phi}$ lies in $G_{-\phi}$ and
is a conjugate of $a$. This yields a contradiction.
\end{proof}

The following example was communicated to the second author by G. Glauberman. It shows that in general $G\neq G_{-\phi}G_\phi$.
\bigskip

{\bf Example.} Let $K$ be the field with $5^3$ elements and let $\phi$ be the automorphism
of order 3 of $K$ sending each $x\in K$ to $x^5$. Denote by $A$ the additive group of $K$ and by $B$ the multiplicative group. Let $G$ be the natural semidirect product $G=AB$ and note firstly that $B$ acts transitively on $A\setminus 0$ and secondly that $\phi$ induces a coprime automorphism of $G$. Obviously $A_\phi$ is a proper subgroup of $A$ and $A_{-\phi}\neq0$. We see that all nontrivial elements of $A_{-\phi}$ are conjugate in $G$ with some elements in $A_\phi$. Lemma \ref {001} shows that $G\neq G_{-\phi}G_\phi$.
\bigskip

Obviously, the group $G$ in the above example is not nilpotent. The next lemma shows that no examples of this kind can be found among nilpotent groups.

\begin{lemma}\label{002} Let $\phi$ be a coprime
automorphism of a finite nilpotent group $G$. Then any
element $x\in G$ can be written uniquely in the form
$x=gh$, where $g\in G_{-\phi}$ and $h\in G_{\phi}$. 
\end{lemma}
\begin{proof}  Assume that $G$ is a counterexample of
minimal possible order. Then, by virtue of Lemma
\ref{001}, we can choose elements $x\in G$ and
$1\neq h\in G_\phi$ such that $x^{-1}x^\phi$ and $h$
are conjugate. By the induction hypothesis $G/Z(G)$
is not a counterexample to the lemma, so it follows
that $h\in Z(G)$. Hence, $x^{-1}x^\phi=h.$ Then in
the natural semidirect product $G\langle\phi\rangle$
we have $$x^{-1}x^\phi\phi^{-1}=h\phi^{-1}.$$ 
Here the element on the left hand side is conjugate to $\phi^{-1}$ and so is of order $n$, where $n$ is the order of $\phi$.
Since  $h\in G_{\phi}$, the element on the right hand side has order $|h|n\neq n$. This is a contradiction.
\end{proof}

\section{Some Lie-theoretic machinery}

Our proof of Theorem \ref{1} relies heavily on the techniques created by Zelmanov in his solution of the restricted Burnside problem. In the present section, for the reader's convenience, we describe some Lie-theoretical tools used in the proof. More information on the subject can be found in \cite{almgt}.

Let $L$ be a Lie algebra over a field ${\mathfrak k}$.
Let $k,n$ be positive integers and let 
$x_1,x_2,\dots,x_k,x,y$ be elements 
of $L$. We define inductively
$$[x_1]=x_1;\ [x_1,x_2,\dots,x_k]=[[x_1,x_2,\dots,x_{k-1}],x_k],$$
and
$$[x,{}_0y]=x;\ [x,{}_ny]=[[x,{}_{n-1}y],y].$$
An element $a\in L$ is
called ad-nilpotent if there exists a positive integer $n$ such that
$[x,{}_na]=0$ for all $x\in L$. If $n$ is the least 
integer with the above property then we say that $a$ is
ad-nilpotent of index $n$. Let $X\subseteq L$ be any subset
of $L$. By a commutator in elements of $X$
we mean any element of $L$ that can be obtained as a Lie
product of elements of $X$ with some system of brackets.
Denote by $F$ the free Lie algebra over 
${\mathfrak k}$ on countably many free generators $x_1,x_2,\dots$.
Let $f=f(x_1,x_2,\dots,x_n)$ be a non-zero element of $F$.
The algebra $L$ is said to satisfy
the identity $f\equiv 0$ if $f(a_1,a_2,\dots,a_n)=0$
for any $a_1,a_2,\dots,a_n\in L$. In this case we say
that $L$ is PI. A deep result of Zelmanov \cite[III(0.4)]{zelm} says that
if a Lie algebra $L$ is PI and is generated by
finitely many elements all commutators in which
are ad-nilpotent then $L$ is nilpotent. A detailed proof of this theorem was published in
\cite{zelmanov}. Using this and some routine
universal arguments, the next theorem can be deduced
(see \cite{khushu}).
\begin{theorem}\label{3.1} Let $L$ be a Lie algebra
over a field ${\mathfrak k}$ generated by 
$a_1,a_2,\dots,a_m$. Assume that $L$ satisfies an
identity $f\equiv 0$ and that
each commutator in the generators
$a_1,a_2,\dots,a_m$ is ad-nilpotent of index 
at most $n$. Then $L$ is nilpotent of 
$(f,m,n,{\mathfrak k})$-bounded class.
\end{theorem}
An important criterion for a Lie algebra to be PI is the
following: 
\begin{theorem}\label{3.2} Let $L$ be a Lie algebra over a field ${\mathfrak k}$.
Assume that a finite group $A$ acts on $L$ by automorphisms in such a manner that $C_L(A)$, the subalgebra formed by fixed elements, is PI. Assume further that the characteristic of ${\mathfrak k}$ is either 0 or prime to the order of $A$. Then $L$ is PI.
\end{theorem}
This theorem was proved by Bahturin and Zaicev for soluble 
groups $A$ \cite{bz} and extended by Linchenko to the general
case \cite{li}. The next result was deduced in \cite{shu11}.
\begin{corollary}\label{3.3} Let $F$ be the free Lie algebra of countable rank over
${\mathfrak k}$. Denote by $F^*$ the set of non-zero elements of $F$. For any finite group $A$ there exists a mapping $$\theta:F^*\rightarrow F^*$$ such that if $L$ and $A$ are
as in Theorem~\ref{3.2}, and if $C_L(A)$ satisfies an identity $f\equiv 0$, then $L$ 
satisfies the identity $\theta(f)\equiv 0$.
\end{corollary}

To be able to use Theorem \ref{3.1} we need a tool
allowing us to deduce that certain elements of $L$ are
ad-nilpotent. In this context the following lemma from \cite{khushu} is
quite helpful.
\begin{lemma}\label{39} Suppose that $L$ is a Lie algebra,
$K$ a subalgebra of $L$ gene\-ra\-ted by $r$ elements
$h_1,\dots ,h_r$ such that all commutators in the $h_i$
are ad-nilpotent in $L$ of index $t$. 
If $K$ is nilpotent of class $c$, then for some
$(c,r,t)$-bound\-ed number $u$ 
we have $[L,\underbrace{K,\dots,K}_{u}]=0$. 
\end{lemma}
We now turn to groups. Throughout the rest of the
section $p$ will denote an arbitrary but fixed prime.
Let $G$ be any group. 
A series of subgroups $$G=G_1\geq G_2\geq\dots   \eqno{(*)}$$
is called an $N_p$-series if $[G_i,G_j]\leq G_{i+j}$ and
$G_i^p\leq G_{pi}$ for all $i,j$. With any $N_p$-series $(*)$ of $G$ one can
associate a Lie algebra $L^*(G)$ over $\mathbb F_p$, the field with $p$ elements.

Namely, view the quotients $L_i^*=G_i/G_{i+1}$ as linear spaces over $\mathbb F_p$, and let  $L^*(G)$ be the direct sum of these spaces. Commutation in $G$ induces a binary operation $[,]$ in $L^{*}(G)$. For homogeneous elements $xG_{i+1}\in L_i^*, yG_{j+1}\in L_j^*$ the operation is defined by $$[xG_{i+1},yG_{j+1}]=[x,y]G_{i+j+1}\in L_{i+j}^*$$ and extended to arbitrary elements of $L^*(G)$ by linearity. Thus, $L^*(G)$ with the operations $+$ and $[,]$ is a Lie algebra over $\mathbb F_p$.

For any $x\in G_i\setminus G_{i+1}$ let $x^*$ denote the element $xG_{i+1}$ of $L^*(G)$.
\begin{proposition}\label{3.6} (Lazard, \cite{la}).
For any $x\in G$ we have $(ad\, x^*)^p=ad\, (x^p)^*$. 
\end{proposition}

The following proposition is immediate from the proof of Theorem 1 in the paper of Wilson and Zelmanov \cite{wize}.
\begin{proposition}\label{3.7} Let $G$ be a group
satisfying a group identity $w\equiv 1$. Then there
exists a non-zero multilinear Lie polynomial $f$ over
$\mathbb F_p$ depending only on $p$ and $w$ such that
for any $N_p$-series $(*)$ of $G$ the
algebra $L^*(G)$ satisfies the identity $f\equiv 0$.
\end{proposition}
In fact Wilson and Zelmanov describe in \cite{wize} an effective algorithm allowing to write $f$ explicitly for any $p$ and $w$ but we do not require this. 

Write $\gamma_i$ for $\gamma_i(G)$, the $i$th term of the lower central series of $G$. Set $D_i=\prod\limits_{jp^k\geq i}\gamma_j^{p^k}$. The subgroups $D_i$ form an $N_p$-series $G=D_1\geq D_2\geq\dots$ in the group $G$. This is known as the Jennings-Lazard-Zassenhaus series.
Let $DL(G)=\oplus L_i$ be the Lie algebra over $\mathbb F_p$ corresponding to the Jennings-Lazard-Zassenhaus series of $G$. Here $L_i=D_i/D_{i+1}$. Let $L_p(G)=\langle L_1\rangle$ be the subalgebra of $DL(G)$ generated by $L_1$. The following result was obtained in Riley  \cite{riley}.
\begin{lemma}\label{4.9} Suppose that $G$ is a finite $p$-group such that the Lie algebra $L_p(G)$ is nilpotent of class $c$. Then $D_{c+1}(G)$ is powerful.
\end{lemma}

Recall that powerful $p$-groups were introduced by 
Lubotzky and Mann in \cite{lbmn}: a finite $p$-group
$G$ is powerful if and only if $G^p\geq [G,G]$ for
$p\ne 2$ (or $G^4\geq [G,G]$ for $p=2$). These groups
have many nice properties. In particular, if such a group
is generated by elements of order dividing $m$, then the exponent
of $G$ divides $m$, too. 

Given a subgroup $H$ of the group $G$, we denote by
$L(G,H)$ the linear span in $DL(G)$ of all homogeneous
elements of the form $hD_{j+1}$, where $h\in D_j\cap H$.
Clearly, $L(G,H)$ is always a subalgebra of $DL(G)$. Moreover,
it is isomorphic with the Lie algebra associated with $H$
using the $N_p$-series of $H$ formed by $H_j=D_j\cap H$.
We also set $L_p(G,H)=L_p(G)\cap L(G,H)$.
Let $\phi$ be an automorphism of the group $G$.
Then $\phi$ acts naturally on every quotient of the
Jennings-Lazard-Zassenhaus series of $G$.
This action induces an automorphism of the Lie
algebra $DL(G)$. So when convenient we will consider
$\phi$ as an automorphism of $DL(G)$ (or of $L_p(G)$).
Lemma \ref{000} implies that if $G$ is finite and
$(|G|,|\phi|)=1$ then $L_p(G,C_G(\phi))=C_{L_p(G)}(\phi)$.
We will also require the following lemma.

\begin{lemma}\label{4.7} Let $G$ be a finite $p$-group and $H$ a
subgroup of $G$. Let $K=L(G,H)$. Then there exists
a number $u$ depending only on the order of $H$ such that
$[DL(G),\underbrace{K,\dots,K}_{u}]=0$.
\end{lemma}
\begin{proof} If $x\in H$, the order of $x$ is at most $|H|$.
Therefore Lazard's Lemma \ref{3.6} shows that the corresponding element
$x^*$ is ad-nilpotent in $DL(G)$ of index at most $|H|$. Furthermore,
it is clear that the nilpotency class of $K$ is at most that of $H$
so the result follows from Lemma \ref{39}.
\end{proof}

\section{Proof of Theorem \ref{1}} 

Recall that Theorem \ref{1} states that if $e$ and $n$ are positive integers and $G$ a finite group admitting a coprime automorphism $\phi$ of order $n$ with the property that any element
of $G_\phi\cup G_{-\phi}$ is contained in a $\phi$-invariant subgroup of exponent dividing $e$, then the exponent of $G$ is $(e,n)$-bounded.

\begin{proof}[Proof of Theorem \ref{1}.] Since $\phi$ leaves invariant some Sylow
$p$-subgroup of $G$ for any prime divisor $p$ of the order of $G$ (Lemma \ref{000}.3), it is sufficient to prove the theorem under the additional assumption that $G$ is a $p$-group. So from now on $G$ is a $p$-group and $e$ is a $p$-power. Since any element
of $G$ is contained in a $\phi$-invariant $n$-generated subgroup,
we can also assume that $G$ is generated by at most $n$ elements.
Let $L=L_p(G)$ and $L_i=(D_i(G)/D_{i+1}(G))\cap L$. We would like to
decompose the linear spaces $L_i$ as direct sums of eigenspaces for $\phi$
but in general this is not possible as the field $\mathbb F_p$
may not contain a primitive $n$th root of unity. To overcome
this difficulty we extend the ground field $\mathbb F_p$.
Let $\omega$ be a primitive $n$th root of unity and set
 ${\overline L}=L\otimes \mathbb F_p [\omega]$.
We will view ${\overline L}$ as a Lie algebra over 
$\mathbb F_p[\omega]$ and $L$ as a subset of ${\overline L}$.
For any $\mathbb F_p$-subspace $S$ of $L$ we write
${\overline S}$ for $S\otimes \mathbb F_p [\omega]$.
The $\mathbb F_p [\omega]$-space ${\overline L}_1$ has a basis
consisting of eigenvectors for $\phi$. It is clear that since $G$ is
$n$-generated, the dimension of ${\overline L}_1$ is at most $n$.
Now fix an arbitrary $i$ and let $a\in{\overline L}_i$ be an
eigenvector for $\phi$. Then $a\in{\overline K}$, where $K$ is some
$\phi$-invariant subspace of $L_i$. We choose $K$ to be the minimal $\phi$-invariant subspace of $L_i$ with the property that $a\in{\overline K}$. Then either $K=K_\phi$ or $K=K_{-\phi}$. It follows that as an $\mathbb F_p\langle\phi\rangle$-module $K$ is generated by a single
element that corresponds to a group element $x$, where $x$ belongs to
either $G_{\phi}$ or $G_{-\phi}$. Let $X$ be a minimal $\phi$-invariant subgroup containing $x$. By the hypothesis the exponent of $X$ is a divisor of $e$. Therefore, by the solution
of the Restricted Burnside Problem \cite{zelm}, the order of $X$ is
$(e,n)$-bounded. Now Lemma \ref{4.7} tells us that there exists
an $(e,n)$-bounded number $u$ such that $[L,\underbrace{K,\dots,K}_{u}]=0$. Of course, this implies that $[{\overline L},\underbrace{{\overline K},\dots,{\overline K}}_{u}]=0$.
In particular, $[{\overline L},{}_{u}a]=0$.

Thus, we have shown that an arbitrary eigenvector contained in
${\overline L}_i$ is ad-nilpotent with $(e,n)$-bounded index at most $u$. It is
clear that if $a$ and $b$ are eigenvectors contained in ${\overline L}_i$
and ${\overline L}_j$ respectively, then  $[a,b]$ is an eigenvector
for $\phi$ contained in ${\overline L}_{i+j}$. So we see that ${\overline L}$
is generated by at most $n$ eigenvectors for $\phi$ and any commutator
in the generators is ad-nilpotent of index at most $u$. 
Using Proposition \ref{3.7} and the fact that $C_L(\phi)=L_p(G,C_G(\phi))$ we conclude that $C_L(\phi)$ satisfies certain $(e,n)$-bounded multilinear polynomial identity. This also holds in $C_{\overline L}(\phi)={\overline {C_L(\phi)}}$. Therefore
Corollary \ref{3.3} implies that $\overline L$ satisfies an $(e,n)$-bounded polynomial identity. It follows from Theorem \ref{3.1} that ${\overline L}$ is nilpotent of $(e,n)$-bounded class, say $c$. The same is true for $L$. Now Lemma \ref{4.9} tells us that 
$H=D_{c+1}(G)$ is powerful. By Lemma \ref{002}, $H=H_\phi H_{-\phi}$.
In particular, $H$ is generated by elements of order dividing $e$. Since
$H$ is powerful, it follows that the exponent of $H$ divides $e$ \cite{lbmn}.
Finally, it is clear that the exponent of $G/H$ divides $p^c$. The
proof is complete.
\end{proof}

\section{Proof of Theorem \ref{2}}

The proof of Theorem \ref{2} will require the following result from \cite{wa} whose proof uses the classification of finite simple groups.

\begin{theorem}\label{wa} If a finite group $G$ admits a coprime automorphism $\phi$ such that $G_\phi$ is nilpotent, then $G$ is soluble. 
\end{theorem}

Recall that the Fitting height $h(G)$ of a finite soluble group $G$ is the minimum number $h$ such that $G$ possesses a normal series of length $h$ all of whose quotients are nilpotent. The next result is a celebrated theorem, due to Thompson \cite{tho} (see \cite{turull} for a survey on results of similar nature).

\begin{theorem}\label{tho} If a finite soluble group $G$ admits a coprime automorphism $\phi$, then $h(G)$ is bounded in terms of the number of prime divisors of the order of $\phi$, counting multiplicities. 
\end{theorem}

The following lemma will be helpful.
\begin{lemma}\label{004} Let $G$ be a finite group admitting a coprime automorphism $\phi$ and assume that $G$ is soluble with derived length $d$. Suppose that $G=[G,\phi]$ and $x^e=1$ for every $x\in G_{-\phi}$. Then $G$ has $(d,e)$-bounded exponent.
\end{lemma}
\begin{proof}	We use induction on $d$. If $d=1$, the result is trivial, so we assume that $d\geq2$. Let $M=G^{(d-1)}$ be the last nontrivial term of the derived series of $G$. By induction, the exponent of $G/M$ is $(d,e)$-bounded. Since $M=M_{\phi}M_{-\phi}$ and $(M_{-\phi})^{e}=1$, taking into account that $M$ is abelian, we deduce that $M^{e} \leq M_{\phi}$. In view of Lemma \ref{0000}, we conclude that $M^{e} \leq Z(G)$. Hence, the exponent of $G/Z(G)$ is $(d,e)$-bounded. A theorem of Mann  \cite{mann} now guarantees that $G'$ has $(d,e)$-bounded exponent. This implies that $M$ has $(d,e)$-bounded exponent too and the result follows. 
\end{proof}

For a subset $X$ of a group admitting an automorphism $\phi$ we write $\langle X\rangle^{\langle\phi\rangle}$ to denote the minimal $\phi$-invariant subgroup containing $X$.
\begin{lemma}\label{005} Let $G$ be a finite nilpotent group admitting a coprime automorphism $\phi$ such that $G=[G,\phi]$. Let $S$ be the set of those elements $h\in G_\phi$ for which there exist $x_1,x_2\in G_{-\phi}$ such that $h\in\langle x_1,x_2\rangle^{\langle\phi\rangle}$. Then $G_{\phi}=\langle S\rangle$.
\end{lemma}
\begin{proof} 	Set $H=\langle S \rangle$. Obviously, $H \leq G_{\phi}$. So we need to prove that $G_{\phi} \leq H$. Let $h \in G_{\phi}$. Since $G=[G,\phi]$, we can write $h=g_1 \cdots g_m$ with $g_i \in G_{-\phi}$. We wish to prove that $h \in H$. This will be shown by induction on $m$. If $m \leq 2$, then obviously $h\in H$, so assume that $m\geq3$. Let $K=\langle g_{m-1},g_{m}\rangle^{\langle\phi\rangle}$. Since $K$ is $\phi$-invariant, according to Lemma \ref{002} we have $g_{m-1}g_{m}=g_0h_0$, where $g_0\in K_{-\phi}$ and $h_0\in K_\phi$. Clearly $K_{\phi} \leq H$ and so $h_0\in H$. Thus, $h=g_1\cdots g_{m-2}g_0h_0$ and $hh_0^{-1}=g_1\cdots g_{m-2}g_0$. By induction we get $hh_{0}^{-1}\in H$, whence $h\in H$. Therefore, we conclude that $G_{\phi}=H$.   
\end{proof}
The next lemma is straightforward from a result of Hartley \cite[Lemma 2.6]{ha}.
\begin{lemma}\label{haha} Let $\phi$ be a coprime automorphism of a finite group $G$. Let $\{N_i : i\in I\}$ be a family of normal $\phi$-invariant subgroups of $G$ and $N=\prod_iN_i$.
Then $N_\phi=\prod_i(N_i)_\phi$.
\end{lemma}

Recall that if $G$ is a nilpotent group of class $c$ which is generated by elements of order dividing $e$, then the exponent of $G$ divides $e^{c}$ (see for example \cite[Corollary 2.5.4]{khukhro}). We are now ready to prove Theorem \ref{2}.

\begin{proof}[Proof of Theorem \ref{2}.] Without loss of generality we may assume that $G=[G,\phi]$. For arbitrary $x_1,x_2\in G_{-\phi}$ let $K(x_1,x_2)$ be the minimal $\phi$-invariant subgroup containing $x_1,x_2$. Lemma \ref{004} tells us that the exponent of $K(x_1,x_2)$ is $(d,e)$-bounded. 

In view of Theorem \ref{wa} the group $G$ is soluble. Thompson's theorem \ref{tho} tells us that the Fitting height of $G$ is bounded by a constant depending only on $n$. The theorem will be proved by induction on the Fitting height of $G$. Suppose first that $G$ is nilpotent. According to Lemma \ref{005} the subgroup $G_\phi$ is generated by intersections $G_\phi\cap K(x_1,x_2)$, where $x_1,x_2$ range over $G_{-\phi}$. Since  $G_{\phi}$ is nilpotent of class $c$, we conclude that the exponent of $G_\phi$ is $(c,d,e)$-bounded. Thus, we see that any element of $G_{\phi}\cup G_{-\phi}$ is contained in a $\phi$-invariant  subgroup of $(c,d,e)$-bounded exponent. Theorem \ref{1} says that the exponent of $G$ is $(c,d,e,n)$-bounded, as required.

Assume now that the Fitting height of $G$ is at least 2. Let $F=F(G)$ be the Fitting subgroup of $G$ and $N=\langle[F,\phi]^{G}\rangle$ be the normal closure of $[F,\phi]$ in $G$. Note that the image of $F$ in $G/N$ is contained in $(G/N)_\phi$ and so, in view of Lemma \ref{0000}, the image of $F$ in $G/N$ is contained in the center $Z(G/N)$. It follows that $h(G/N)\leq h-1$ and therefore by induction the exponent of $G/N$ is $(c,d,e,n)$-bounded. Hence, it is sufficient to show that the exponent of $N$ is $(c,d,e,n)$-bounded. Clearly, $N=N_{\phi}[F, \phi]$. We already know that the exponent of $[F,\phi]$ is $(c,d,e,n)$-bounded. In view of Theorem \ref{1} it remains to show that the exponent of $N_{\phi}$ is also bounded.

The subgroup $N$ is generated by the conjugates $[F,\phi]^x$, where $x$ ranges over $G$. For each $x\in G$ let $N_x$ be the minimal $\phi$-invariant subgroup of $G$ containing $[F,\phi]^x$. Since the exponent of  $[F,\phi]$ is $(c,d,e,n)$-bounded and $N_x$ is a product of at most $n$ normal subgroups of the form  $([F,\phi]^x)^{\phi^i}$, we conclude that the exponent of $N_x$ is $(c,d,e,n)$-bounded, too. In particular, the exponent of $(N_{x})\cap G_{\phi}$ is $(c,d,e,n)$-bounded. Lemma \ref{haha} says that $N_{\phi}$ is the product of subgroups $(N_{x})\cap G_{\phi}$, where $x\in G$. Thus, $N_{\phi}$ is a nilpotent group of class at most $c$ generated by elements of $(c,d,e,n)$-bounded order. Hence, the exponent of $N_{\phi}$ is $(c,d,e,n)$-bounded and this implies that the exponent of $N$ is bounded in terms of $c,d,e,n$ only. The proof is now complete. 
\end{proof}

\end{document}